\documentclass[12pt]{article}
\usepackage[T1]{fontenc}
\usepackage{mathtools}
\usepackage{wesu}
\usepackage{tipa}
\usepackage{amssymb}
\usepackage{tikz}
\usepackage{bbm}
\usepackage[style=nature, backend=biber]{biblatex}
\usetikzlibrary{automata}
\usepackage{latexsym}
\usepackage{amsmath,amssymb,amsthm}

\makeatletter
\usepackage{setspace}
	\newcommand{\MSonehalfspacing}{
  		\setstretch{1.44}
  			\ifcase \@ptsize \relax
   		 \setstretch {1.448}
  			\or
   		 \setstretch {1.399}
  			\or
    	 \setstretch {1.433}
  			\fi								}				
	\newcommand{\MSdoublespacing}{
 		\setstretch {1.92}
 			\ifcase \@ptsize \relax
    	\setstretch {1.936}
  			\or
    	\setstretch {1.866}
  			\or
    	\setstretch {1.902}
  			\fi								}
	\MSonehalfspacing
\makeatother
\newtheorem{Satz}{Satz}[section]
\newtheorem{Definition}[Satz]{Definition}     
\newtheorem{Lemma}[Satz]{Lemma}	
\newtheorem{Remark}[Satz]{Remark}	
  
\newtheorem{Theorem}[Satz]{Theorem}                   
\numberwithin{equation}{section} 
\newcommand{\R}{\mathbb{R}} 
\newcommand{\N}{\mathbb{N}}
\newcommand*{\defeq}{\mathrel{\vcenter{\baselineskip0.5ex \lineskiplimit0pt
                     \hbox{\scriptsize.}\hbox{\scriptsize.}}}%
                     =} 

\begin{document}
\pagestyle{plain}
\title{A Constructive Winning Maker Strategy in the Maker-Breaker $C_4$-Game}
\author{Matthias Sowa and Anand Srivastav \\ %
	Department of Mathematics, Kiel University,\\
	Boschstr. 1, 24118 Kiel, Germany;\\
	\{sowa,srivastav\}@math.uni-kiel.de
}
\date{}
\maketitle

\begin{abstract}
Maker-Breaker subgraph games are among the most famous combinatorial games. For given $n,q\in \mathbb{N}$ and a subgraph $C$ of the complete graph $K_n$, the two players, called Maker and Breaker, alternately claim edges of $K_n$. In each round of the game Maker claims one edge and Breaker is allowed to claim up to $q$ edges. If Maker is able to claim all edges of a copy of $C$, he wins the game. Otherwise Breaker wins. In this work we introduce the first constructive strategy for Maker for the $C_4$-Maker-Breaker game and show that he can win the game if $q<0.16 n^{2/3}$. According to the theorem of Bednarska and Łuczak (2000) $n^{2/3}$ is asymptotically optimal for this game, but the constant given there for a random Maker strategy is magnitudes apart from our constant $0.16$.
\end{abstract}

\section{Introduction}
\subsection{Previous work}
Bednarska and Łuczak \cite{Bednarska-Luczak} proved that, if $C$ contains three nonisolated vertices, there exists constants $c_1,c_2>0$ such that for sufficiently large $n$ Maker can win, if $q\leq c_1 q^{1/m(C)}$ and Breaker can win, if $q\geq c_2 q^{1/m(C)}$, where $m(C):= max \big\{  \frac{\vert E(H)\vert-1}{\vert V(H)\vert-1}: H \text{ is a subgraph of }C,\,\vert V(H)\vert \geq 3 \big\}$. They conjectured that $c_1$ and $c_2$ could be choosen arbitrarily near to each other but this statement is not yet proven for any $C$ containing a circle. In the case of $C=C_3$ Chvátal and Erdős \cite{Chvatal-Erdos} showed that Maker can win the game if $q < \sqrt{2}\sqrt{n}$ and Glazik and Srivastav \cite{Glazik} gave a winning strategy for Breaker for $q> \sqrt{8/3}\sqrt{n}$. The techniques invented there can be generalized to prove that Breaker has a winning strategy in the $C_4$-game if $q>1.89 n^{2/3}$. Till now there was no constructive and deterministic strategy known for Maker and the constant $c_1$ for the random strategy for Maker in \cite{Bednarska-Luczak} can only be estimated by $0.81892\cdot10^{-6}$ (see \cite{Wolos}).

\subsection{Our contribution}
We will introduce a generalization of the minimum degree game which we call the partial minimum degree game and state a Maker strategy that wins the game within a certain time restriction. The partial minimum degree game is strongly connected to the $C_4$-Maker-Breaker game and we will show that Maker can win the $C_4$-game by playing our strategy for the partial degree game until a certain point of time.

\section{Maker's Strategy and it's Analysis}
\subsection{The Partial Minimum Degree Game}
In this chapter we introduce a generalization of the Maker-Breaker minimum degree game. In the minimum degree game Maker's goal is it to claim a subset $M$ of edges of $K_n=(V,E)$ such that $deg_M(v)\geq d(n)$ for every vertex of $v \in V$, where $d$ is a fixed function. Our generalization, which we call the partial minimum degree game, is the following: Given $\beta \in (0,1]$ Maker's goal is it to claim a subgraph $M$ such that there is a subset $X \subseteq V$ with $\vert X \vert \geq n \beta$ and $deg_M(x)\geq d(n)$ for every vertex $x\in X$.
Let us define the game formally.

\begin{Definition} 
	Let $\beta \in (0,1]$, $d:\N \to \R$, and 
	\begin{align*}
		\mathcal{F}_{\beta,d(n)} \defeq \bigg\{ M\subseteq E(K_n) \big\vert \exists X\subseteq V(K_n):|X|\geq \beta n  
		\land \forall x\in X \, deg_M(x) \geq d(n)\bigg\}.
	\end{align*}
	We call the Maker-Breaker game with the winning sets $\mathcal{F}_{\beta,d(n)}$ the $(\beta,d(n))$-partial minimum degree game.
\end{Definition}

For $\beta =1$ the $(1,d(n))$-partial minimum degree game is just the well-known minimum degree game. 
Beck \cite{Beck} showed that Maker has a winning strategy in the minimum degree game if $d(n)=\frac{1-\epsilon}{q(n)+1}n$, where 
$$\epsilon = \bigg( 1+O\bigg( q\sqrt{\frac{\log(n)}{(q+1)(n-1)}} \bigg) \bigg)2q\sqrt{\frac{\log(n)}{(q+1)(n-1)}}\,.$$
The following Theorem for the partial minimum degree game is similar to the result of Beck, but in our setting $\beta<1$ and this means that it is sufficient for Maker to achieve a high minimal degree only on a subset of vertices. Due to this we are able not only to give a much simpler proof but also to show that the number of rounds Maker needs to win is at most $\frac{dn}{2}$, where $d$ is the target degree. In fact, this bound on then number of turns is the key to tackle the $C_4$-game.

For any $d>0$ we state the $d$-degree strategy for Maker: In each of his turns Maker claims an unclaimed edge $\{v,w\}$ that fulfills $deg_M(v)<d$ and $deg_M(w)<d$, where $M$ is the graph of Maker's current edges. If he cannot do so, he stops. We will show that at this point Maker has already won the the partial minimum degree game.

\begin{Theorem} \label{theorem:pmd-game}
	Let $\delta>0$, $\alpha,\beta\in ]0,1[$, and $d:\N \to \R: n \mapsto  \delta n^{1-\alpha}$ .\\ We consider the partial minimum degree game with regard to $\mathcal{F}_{\beta,d(n)}$ played on the edges of $K_n$.
	Let $d=d(n)=\delta n^{1-\alpha}$.
	For sufficiently large $n \in \N$ and a bias $q<\frac{(1-\beta)^2}{\delta}n^\alpha$ Maker wins in at most $\frac{\delta}{2}n^{2-\alpha}$ turns by playing according to the $d$-degree strategy.
\end{Theorem}
\begin{proof}
	Let $c<\frac{(1-\beta)^2}{\delta}$, $q=cn^\alpha$ and let $n\in\N$ be large enough so that the inequality $\frac{c\delta}{2}n^2<\frac{(1-\beta)^2}{2}n^2-\frac{1}{2}\big((1-\beta) n+\delta(1-\beta) n^{2-\alpha}\big)$ holds.\footnote{This is possible since $c\delta<(1-\beta)^2$ and the term $\frac{1}{2}\big((1-\beta) n+\delta(1-\beta) n^{2-\alpha}\big)$ is negligible for large $n$.} Maker plays according to the $(\delta n^{1-\alpha})$-degree strategy which means that he only claims edges $\{v,w\}$ with $deg_M(v)<\delta n^{1-\alpha}$ and $deg_M(v)<\delta n^{1-\alpha}$ as long as he is able to do so, where $M$ is the graph of Maker's edges.
	Note that after Maker has claimed more than $\frac{\delta}{2}n^{2-\alpha}$ edges we have $\sum_{v \in V}deg_M(v) > \delta n^{2-\alpha}$ and thus, by the pigeonhole principle, there has to be a vertex vertex $v$ with $deg_M(v)>\delta n^{1-\alpha}$.
	Therefore we know that after at most $\frac{\delta}{2}n^{2-\alpha}$ turns there is no unclaimed edge $\{v,w\}$ left with $deg_M(v)<\delta n^{1-\alpha}$ and $deg_M(w)<\delta n^{1-\alpha}$.
	Let be $M$ and $B$ the sets of Maker- and Breaker-edges respectively at this moment, and define $X\defeq \{ v\in V:\,deg_M(v)\geq\delta n^{1-\alpha} \}$ and $Y\defeq V\setminus X$. We assume for a moment that the assertion of the theorem is not true. Thus by Definition 2.1 of the winning sets, $|X|< \beta n$, therefore $|Y|\geq (1-\beta) n$. For each $y\in Y$, we have $V\setminus\{y\}\subseteq X \cup N_M(y) \cup N_B(y)$ because otherwise for $v \in \big(V\setminus\{y\}\big)\setminus \big(X \cup N_M(y) \cup N_B(y)\big)$ the edge $\{v,y\}$ would be still unclaimed, which is a contradiction to the assumption that Maker cannot claim an edge containing two vertices with degree smaller than $\delta n^{1-\alpha}$. Therefore, $n-1\leq |X|+|N_M(y)|+|N_B(y)|$ which implies $(1-\beta)n-1-\delta n^{1-\alpha}\leq|N_B(y)|$. This means, Breaker has occupied at least $(1-\beta)n-1-\delta n^{1-\alpha}$ edges incident in every vertex of $Y$. So in total the number of edges claimed by Breaker is at least
\begin{align*}
	&\frac{\vert Y \vert}{2}\big((1-\beta) n-1-\delta n^{1-\alpha}\big)\\
	\geq&\frac{(1-\beta)n}{2}\big((1-\beta) n-1-\delta n^{1-\alpha}\big)\\
	=&\frac{(1-\beta)^2}{2}n^2-\frac{1}{2}\big((1-\beta) n+\delta(1-\beta) n^{2-\alpha}\big)\\
	>&\frac{c\delta}{2}n^2 =\frac{\delta}{2}n^{2-\alpha} q
\end{align*}
which is a contradiction to the fact that at most $\frac{\delta}{2}n^{2-\alpha}$ rounds were played.	
\end{proof}

\subsection{A Winning Strategy for Maker for the $C_4$-Game}
In this section we will give a constructive, asymptotically optimal Maker strategy for the Maker-Breaker $C_4$-Game. We will use Theorem \ref{theorem:pmd-game} for the partial minimum degree game to establish a new lower bound for the threshold bias for the $C_4$-Game.

\begin{Theorem}\label{theorem:C4}
	Let $c <0.16$ and $q= q(n) = cn^{2/3}$ . Then for sufficiently large $n$ there is a winning strategy
	for Maker in the Maker-Breaker-$C_4$-Game.
\end{Theorem}

For the proof of the Theorem, we fix $\delta >1$ with $c\delta< 0.16$ and \\${\textstyle \beta \in \textstyle (0.6,1-\sqrt{c\delta})}$ and assume that  $n$ is sufficiently large. The parameters $c$, $\delta$, $q$, and $\beta$ are fixes for the remainder of this work.\\
\textit{\textbf{Maker's strategy for the $C_4$-game}:
Maker plays according to the $\delta n^{1/3}$-degree strategy. }\\
Note that by Theorem \ref{theorem:pmd-game} he wins the $(\beta,\delta n^{2/3})$-partial minimum degree. We now show that Maker needs at most two additional turns to win the $C_4$-game. First we need a few definitions to analyze the game state at the point of time, when Maker would win the partial minimum degree game. We assume that at this point of time Maker has not yet won the $C_4$-Game and show that he needs only two more turns to win it.

\begin{Definition}\label{definition: dangerous}
	Let $M, B\subseteq E(K_n)$ be Maker's graph and Breaker's graph respectively, and let $X\defeq \{ v\in V:\,deg_M(v)\geq\delta n^{1/3} \}$.
	\begin{enumerate}
		\renewcommand{\labelenumi}{(\roman{enumi})}
		\item For each $a =\{a_1,a_2\}\in E$ we call the edges in
		$$T_a \defeq \big\{ \{b_1,b_2\}\in E : b_1\in N_M(a_1) \land b_2\in N_M(a_2) \big\}\setminus \{a\}$$
		the threats for Breaker induced by the edge $a$.
		\item $D \defeq \big\{ e\in E: \vert T_e \vert \geq \delta^2 n^{2/3}-1\big\}$ is the set of dangerous edges (for Breaker).
		\item $D_d \defeq D\cap B$ is the set of directly deactivated edges.
		\item $D_i \defeq \big\{ e\in D: \vert T_e \setminus B \vert \leq q\big\}$ is the set of indirectly deactivated edges.
		\item $D_a \defeq D\setminus(D_d \cup D_i)$ is the set of active dangerous edges.
	\end{enumerate}
	\begin{center}
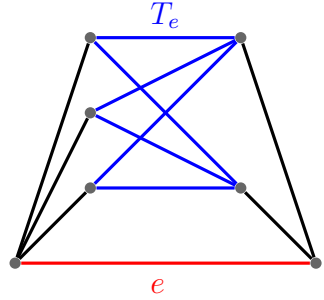
  		\begin{minipage}{\linewidth}
		    \centering
  		\begin{tikzpicture}
			\useasboundingbox (-4.96,6.8) rectangle (100mm,2.5);
			\tikzstyle{blind} =[circle,fill=black!0, inner sep=2pt]
			\tikzstyle{vertex} =[circle,fill=black!60, inner sep=1.5pt]
			\tikzstyle{edgeG} =[very thick]
			\tikzstyle{edgeM} =[very thick, red!100]
			\tikzstyle{edgeB} =[very thick, blue!100]
			\node[vertex](v) at (0,3){};
			\node[vertex](w) at (4,3){};
			\node[vertex](v1) at (1,4){};
			\node[vertex](v2) at (1,5){};
			\node[vertex](v3) at (1,6){};
			\node[vertex](w1) at (3,4){};
			\node[vertex](w2) at (3,6){};
			\draw[edgeM] (v)--(w);
			\draw[edgeG] (v)--(v1);
			\draw[edgeG] (v)--(v2);
			\draw[edgeG] (v)--(v3);
			\draw[edgeG] (w)--(w1);
			\draw[edgeG] (w)--(w2);
			\draw[edgeB] (v1)--(w1);
			\draw[edgeB] (v2)--(w1);
			\draw[edgeB] (v3)--(w1);
			\draw[edgeB] (v1)--(w2);
			\draw[edgeB] (v2)--(w2);
			\draw[edgeB] (v3)--(w2);
			\node[text width=5cm, blue!100] at (4.3,6.3) {$T_e$};
			\node[text width=5cm, red!100] at (4.3,2.7) {$e$};
		\end{tikzpicture}\\
		\end{minipage}
	\end{center}
\end{Definition}

Due to Maker's strategy, we have $deg_M(v)=\delta n^{2/3}$ for each vertex $v\in X$. Maker can create a $C_4$ if he claims an edge $e$ and also an edge from $T_e$. Because Breaker is allowed to claim $q$ edges in each turn, Maker has to claim an edge that induces more than $q$ threats in one of his turns in order to win the game. The edges of $D$ could induce the highest number of threats and are therefore called dangerous. Breaker can stop Maker from claiming a dangerous edge by claiming the edge himself and thereby deactivating the danger of the edge. The set of edges that are dangerous but deactivated by Breaker in this direct way is $D_d$. The second way for Breaker to deactivate a dangerous edge $e$ is claiming enough edges of $T_e$ such that there are no more than $q$ edges left in $T_e$ unclaimed by Breaker. In this case Breaker would be able to claim all remaining threats if Maker claims $e$. The set of the edges that are deactivated in this indirect way is $D_i$. The set of edges that are dangerous and neither deactivated directly nor indirectly is $D_a$. These are the active dangerous edges and if at any point of time Maker is able to claim one of these he will present more threats than Breaker can deny and Maker will win in the next turn.

\begin{Remark}\label{remark: switching edges}
	It is evident that for all $e \in E$ the following two statements hold.
	\begin{enumerate}
		\renewcommand{\labelenumi}{(\roman{enumi})}
		\item $e \in T_a \Leftrightarrow a\in T_e$ for all $a\in E$.
		\item $T_e = \{ a\in E : e\in T_a \}$.
	\end{enumerate}
\end{Remark}

\begin{Lemma}\label{lemma: threats}
	Let $e=\{v_1,v_2\}\in E$ be an edge with $d \defeq deg_M(v_1)=deg_M(v_2)$. We assume that $M$ does not contain a copy of $C_4$.\\Then $\vert T_e \vert \in \{ d^2-1,d^2 \}$.
\end{Lemma}
\begin{proof}
	If $N_M(v_1)\cap N_M(v_2)=\emptyset$, we have $\vert T_e \vert = deg_M(v_1)\,deg_M(v_2)=d^2$.
	Otherwise let $x,y \in N_M(v_1)\cap N_M(v_2)$. Since $M$ does not contain a copy of $C_4$, the closed walk $(v_1,x,v_2,y,v_1)$ cannot be a $C_4$, therefore $x=y$. Thus $\vert N_M(v_1)\cap N_M(v_2)\vert =1$. It follows 
	\begin{align*}
	\vert T_e \vert 
	&= \vert N_M(v_1) \setminus \{x\}\vert \, \vert N_M(v_2) \vert + \vert N_M(v_2) \setminus \{x\}\vert\\ 
	&= (d-1)d+(d-1)=d^2-1
	\end{align*}
\end{proof}

\begin{Lemma}\label{lemma:dangerous edges}
	If $n$ is sufficiently large and $M$ doesn't contain a copy of $C_4$, the following statements hold:
	\begin{enumerate}
		\renewcommand{\labelenumi}{(\roman{enumi})}
		\item $\vert B \vert \leq \frac{c \delta}{2}n^2$
		\item $\vert D \vert \geq \binom{\beta n}{2}$
		\item $D \setminus D_d \subseteq D_i \cup D_a$
		\item $\vert D \setminus D_d \vert \geq \binom{\beta n}{2}-\frac{c \delta}{2}n^2$
	        	\item $\vert B \vert \geq \vert D_i \vert (1-\frac{c}{\delta^2}-n^{-2/3})$
		\item $\vert D_a \vert > 0$
	\end{enumerate}
\end{Lemma}

\begin{proof}
\begin{enumerate}
\renewcommand{\labelenumi}{(\roman{enumi})}
\item By Theorem \ref{theorem:pmd-game},
Maker can win the $(\beta,\delta n^{2/3})$-partial minimum degree game in at most $\frac{\delta}{2}n^{4/3}$ rounds, so Breaker has claimed at most $q \frac{\delta}{2}n^{4/3}= \frac{c \delta}{2}n^2$ edges.
\item Since $deg_M(v)\geq \lceil \delta n^{1/3} \rceil$ for all $v\in X$, by Lemma \ref{lemma: threats} we have $\vert T_e \vert \geq \delta^2 n^{2/3}-1$ for each edge $e\in E$ with $e\subseteq X$. Thus $\binom{X}{2}\subseteq D$. By Theorem \ref{theorem:pmd-game}, we have $\vert X \vert \geq \beta n$ and therefore $\vert D \vert \geq \binom{\vert X \vert}{2} \geq \binom{\beta n}{2}$.
\item Since $D_a = D\setminus(D_d \cup D_i)$, the statement is evident.
\item We know from (i) that $\frac{c \delta}{2}n^2 \geq \vert B \vert \geq \vert B\cap D \vert = \vert D_d \vert$. The statement follows from (ii).
\item Since Maker plays by the $\delta n^{1/3}$-degree strategy, each vertex has Maker degree of at most $\delta n^{1/3}$, which implies $\vert T_e \vert \leq \delta^2 n^{2/3}$ for each $e\in E$.\\
It follows
\begin{flalign*}
	\textstyle \vert D_i \vert \big((\delta^2-&c)n^{2/3}-1\big)\\
	&= \sum_{e \in D_i}\big((\delta^2-c)n^{2/3}-1\big)\\
	&= \sum_{e \in D_i}(\delta^2n^{2/3}-1-q)\\
	&\leq \sum_{e \in D_i}(\vert T_e\vert-q)& (\text{as }D_i\subseteq D)\\
	&\leq \sum_{e \in D_i}(\vert T_e \cap B \vert + \vert T_e \setminus B\vert-q)\\
	&\leq \sum_{e \in D_i}\vert T_e \cap B \vert& (\text{by Definition } \ref{definition: dangerous}(iv))\\
	&= \sum_{e \in D_i}\vert \{b\in B:b\in T_e \}\vert\\
	&= \sum_{e \in D_i}\vert \{b\in B:e\in T_b \}\vert& (\text{by Remark }\ref{remark: switching edges})\\
	&= \sum_{e \in D_i}\sum_{b \in B} \mathbbm{1}_{T_b}(e) = \sum_{b \in B}\sum_{e \in D_i} \mathbbm{1}_{T_b}(e)\\
	&= \sum_{b \in B}\vert T_b \cap D_i \vert \leq \sum_{b \in B}\vert T_b\vert\\
	&\leq \sum_{b \in B}\delta^2n^{2/3} = \vert B \vert \delta^2n^{2/3}\,.
\end{flalign*}
Hence $\vert B \vert \geq \vert D_i \vert \frac{(\delta^2-c)n^{2/3}-1}{\delta^2 n^{2/3}} \geq  \vert D_i \vert (1-\frac{c}{\delta^2}-n^{-2/3})$.
\item We assume $D_a = \emptyset$. Because $\beta>0.6$, the following inequality is true:
\begin{equation*}
	(\beta^2-(1-\beta)^2)(1-(1-\beta)^2)>(1-\beta)^2\,.
\end{equation*}
Therefore, for sufficiently large n, 
\begin{equation}\label{equation: beta}
	(\beta^2-(1-\beta)^2)(1-(1-\beta)^2-n^{-2/3})>(1-\beta)^2\,.
\end{equation}
We give upper and lower bounds for $\vert B \vert$:
\begin{flalign*}
	\textstyle \frac{c\delta}{2} n^2 &\geq \textstyle \vert B\vert &(\text{by } (i))\\
	&\geq \textstyle \vert D_i \vert (1-\frac{c}{\delta^2}-n^{-2/3}) &(\text{by } (v))\\
	&\geq \textstyle \vert D \setminus D_d\vert (1-\frac{c}{\delta^2}-n^{-2/3})& (\text{by } (iii))\\
	&\geq \textstyle \big( \binom{\beta n}{2}-\frac{c \delta}{2}n^2\big) (1-\frac{c}{\delta^2}-n^{-2/3}) &(\text{by } (iv))\\
	&= \textstyle \big(\frac{1}{2}(\beta^2-c\delta)(1-\frac{c}{\delta^2}-n^{-2/3})\\
	&\hspace{44.6pt} \textstyle -\frac{\beta}{2n}(1-\frac{c}{\delta^2}-n^{-2/3}) \big)n^2\\
	&\geq \textstyle \big(\frac{1}{2}(\beta^2-c\delta)(1-c\delta-n^{-2/3})\\
	&\hspace{44.6pt} \textstyle -\frac{\beta}{2n}(1-\frac{c}{\delta^2}-n^{-2/3}) \big)n^2 &(\text{as } \delta > 1)\\
	&\geq \textstyle \big(\frac{1}{2}(\beta^2-(1-\beta)^2)(1-(1-\beta)^2-n^{-2/3})\\
	&\hspace{44.6pt} \textstyle -\frac{\beta}{2n}(1-\frac{c}{\delta^2}-n^{-2/3})\big)n^2 &(\text{as } c\delta<(1-\beta)^2)\\
	&>\textstyle \big( \frac{(1-\beta)^2}{2}-\frac{\beta}{2n}(1-\frac{c}{\delta^2}-n^{-2/3})\big)n^2\,.   &(\text{with }(\ref{equation: beta}))
\end{flalign*}
The above chain of inequalities gives
\begin{align}\label{equation: c_delta>...}
	c \delta &\geq (1-\beta)^2-\textstyle \frac{\beta}{n}(1-\frac{c}{\delta^2}-n^{-2/3}) &\notag\\
	&=(1-\beta)^2 -o(1)\,.
\end{align}
Since $c\delta<(1-\beta)^2$, (\ref{equation: c_delta>...}) cannot hold for sufficiently large $n$. Thus the assumption $D_\alpha =\emptyset$ is false, and we have proved (vi).
\end{enumerate}
\end{proof}

\begin{proof}[Proof of Theorem \ref{theorem:C4}]
	We have shown in Lemma \ref{lemma:dangerous edges}(v) that $D_a \neq \emptyset$. To win the $C_4$-game Maker can now claim one of these active dangerous edges $e \in D_a$. Since $e\notin D_i$, we have $\vert T_e \vert \setminus B > q$ and because Breaker could claim only $q$ edges, in Maker's next turn at least one of the edges of $T_e$ will still be not claimed and Maker completes a copy of $C_4$ by claiming one of them.
\end{proof}

\printbibliography

\subsection*{Acknowledgement}
We would like to thank Małgorzata Bednarska-Bzdęga (Adam Mickiewicz University, Poznań) for helpful comments on this paper.

\end{document}